%% file: arxiv-SSQ-Douglas.tex
\documentclass[10pt,a4paper]{article}
\usepackage{amsmath}
\usepackage{latexsym,amssymb,amsmath,amsthm,amsfonts}
\usepackage{graphicx}
\usepackage{enumerate}

\allowdisplaybreaks[4]

\newtheorem{lemma}{Lemma}[section]
\newtheorem{proposition}[lemma]{Proposition}
\newtheorem{theorem}[lemma]{Theorem}

\newtheorem{example}[lemma]{Example}

\numberwithin{equation}{section}

\input{symbol}

\makeatletter

\newcommand{\Rmnum}[1]{\expandafter\@slowromancap\romannumeral #1@}
\makeatother

\begin{document}
\title{On singular square metrics with vanishing Douglas curvature}

\author{Changtao Yu and Hongmei Zhu}
\date{2016.~11.~1}
\maketitle

\begin{abstract}
Square metrics $F=\sq$ are a special class of Finsler metrics. It is the rate kind of metric category to be of excellent geometrical properties. In this paper, we discuss the so-called singular square metrics $F=\ssq$. A characterization for such metrics to be of vanishing Douglas curvature is provided. Moreover, many analytical examples are achieved by using a special kinds of metrical deformations called $\b$-deformations.
\end{abstract}

\section{Introduction}
A (regular) Finsler metric $F$ on a manifold $M$ is a homogeneous continuous function $F:TM\rightarrow[0,+\infty)$ where $F$ is smooth  on the slit tangent bundle $TM_o$ satisfying nonnegativity~($F(y)>0$ for any $y\neq0$) and strong convexity~(the fundamental tensor $g_{ij}:=[\frac{1}{2}F^2]_{y^iy^j}$ is positive definite on $TM_o$). Here $(x^i,y^i)$ denote the natural system of coordinates of $TM$.

In 1929, L.~Berwald constructed the following famous Finsler metric\cite{Be}
\begin{eqnarray*}
F=\f{(\sqrt{(1-|x|^2)|y|^2+\langle x,y\rangle^2}+\langle x,y\rangle)^2}{(1-|x|^2)^2\sqrt{(1-|x|^2)|y|^2+\langle x,y\rangle^2}}.
\end{eqnarray*}
This metric, defined on the unit ball $\mathbb B^n(1)$ with all the straight line segments as its geodesics, has constant flag curvature $K=0$. In a modern point of view, Berwald's metric belongs to a special kind of Finsler metrics called Berwald type metrics or square metrics given as the form
\begin{eqnarray}\label{sq}
F=\sq,
\end{eqnarray}
where $\a$ is a Riemannian metric and $\b$ is a $1$-form\cite{szm-yct-oesm}. It is known that (\ref{sq}) is a regular Finsler metric if and only if the length of $\b$ with respect to $\a$, denoted by $b$, satisfies $b<1$.

The aim of this paper is to study the following kind of Finlser metrics,
\begin{eqnarray}\label{ssq}
F=\ssq.
\end{eqnarray}
(\ref{ssq})  corresponds to (\ref{sq}) with $b=1$ just by taking $\ba=b^2\a$ and $\bb=b\b$~(in this case, $F=\frac{(\ba+\bb)^2}{\ba}$ with $\bar b\equiv1$). See Proposition \ref{imineiigadiing} and the related discussions below it for details. (\ref{ssq}) are not regular obviously, but the singularity is very weakly. Actually, by (\ref{ppp}) and (\ref{ppp1}) one can see that the singularity shows up as that $F(y)$ takes the value zero along a single direction determined by $b\a+\b=0$ and $(g_{ij})$ as a matrix is degenerate on two directions determined by $b\a\pm\b=0$. Hence, we will called them {\emph singular square metrics}. Figure 1 and Figure 2 show the difference between the indicatriaxs of regular and singular square metrics.

\begin{figure}[h]
 \begin{minipage}[t]{0.5\linewidth}
  \centering
  \includegraphics[scale=0.7]{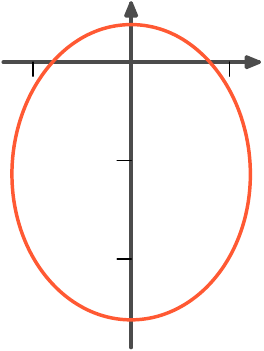}
  \caption{Indicatrix of square metrics with $b<1$}\label{ea}
  \end{minipage}%
\begin{minipage}[t]{0.5\linewidth}
  \centering
   \includegraphics[scale=0.7]{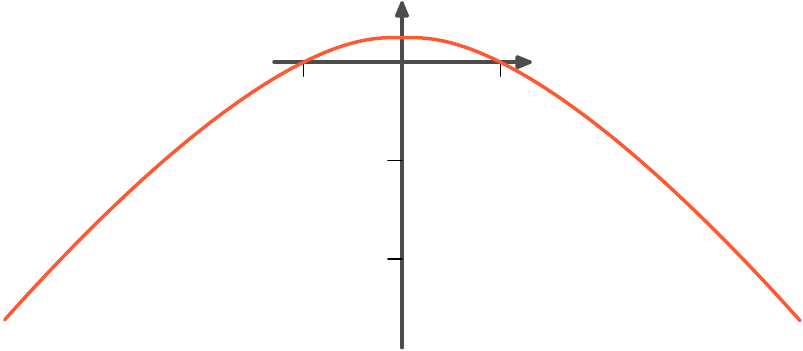}
  \caption{Indicatrix of square metrics with $b=1$}\label{eb}
  \end{minipage}
\end{figure}

Square metrics form a rare kind of Finsler metrics to be of great geometric properties. Recently, E. Sevim et al. studied Einstein $\ab$-metrics~(a special Finsler metrical category in which all the metrics are determined by a Riemannian metric $\a$ and a $1$-form $\b$ just like square metrics) of Douglas type, and proved that such Finsler metrics must be either a Randers metric or a square metric\cite{SSZ}.  Later on, Z. Shen and the first author provided a concise characterization for Einstein square metric\cite{szm-yct-oesm}.

Recently, the authors study general $\ab$-metrics~(a larger metrical category including $\ab$-metrics), and we accidentally find that except for regular square metrics, singular square metrics are also of wonderful Riemann curvature properties\cite{yct-zhm-pfga}. This motive us to study the non-Riemann curvature properties of singular square metrics.

In this paper, we will focus on Douglas curvature, a typical non-Riemann curvature for Finsler metrics introduced by J. Douglas in 1927\cite{D,BM}. We provide a characterization for singular square metrics to be of vanishing Douglas curvature~(such metrics are called {\emph Douglag metrics}) below.

\begin{theorem}\label{characterizationforssqD}
Let $F=\frac{(b\a+\b)^2}{\a}$ be a singular square metric on a $n$-dimensional manifold $M$ with $n\geq3$. Then $F$ is a Douglas metric if and only if $\a$ and $\b$ satisfy
\begin{eqnarray}
r_{ij}&=&c(x) a_{ij}+d(x) b_{i}b_{j}-\frac{3}{b^2}(s_{i}b_{j}+s_{j}b_{i}),\label{ssqrij}\\
s_{ij}&=&\frac{1}{b^2}(b_{i}s_{j}-b_{j}s_{i}).\label{ssqsij}
\end{eqnarray}
\end{theorem}

The meaning of the related symbols can be found in Section 2.

We guess that the ratio $c(x):d(x)$ depends only on $b$. It's true for all the regular Douglas $\ab$-metrics including regular square metrics\cite{yct-dmab}. But we have no idea how to proof it at this moment. It is important. Because if it is true, then the following theorem will be the classification result for singular Douglas square metrics, which is the summarization of Lemma \ref{ieminignienigss}, Proposition \ref{meimeinignni}, Lemma \ref{iemimaubgeg} and Proposition \ref{iwienignubhd}.
\begin{theorem}\label{Douglasssq}
Assume $c(x):d(x)$ depends only on $b$. Then a singular square metric (\ref{ssq}) is a Douglas metric if and only if it can be reduced to one of the following cases:
\begin{enumerate}[(i)]
\item $\a$ and $\b$ can be expressed as
\begin{eqnarray*}
\a=\sqrt{\ba^2+(\bar b^{-8}-\bar b^{-2})\bb^2},\qquad\b=\bar b^{-3}\bb,
\end{eqnarray*}
where the Riemiannian metric $\ba$ and the $1$-form $\bb$ satisfy
\begin{eqnarray}
\brij=\bar\tau\baij,\qquad\bsij=\f{1}{\bar b^2}(\bbi\bsj-\bbj\bsi).\label{specialoneformssqDa}
\end{eqnarray}
In this case, $\bar b=b$ and $F$ is reexpressed as
\begin{eqnarray}\label{iemiingiiiwa}
F=\f{(\sqrt{\bar b^6(\bar b^2\ba^2-\bb^2)+\bb^2}+\bb)^2}{\bar b^2\sqrt{\bar b^6(\bar b^2\ba^2-\bb^2)+\bb^2}};
\end{eqnarray}
\item $\a$ and $\b$ can be expressed as
\begin{eqnarray*}
\a=\bar b^{-3}\ba,\qquad\b=\bar b^{-3}\bb
\end{eqnarray*}
where the Riemiannian metric $\ba$ and the $1$-form $\bb$ satisfy
\begin{eqnarray*}
\brij=\bar\tau(\bar b^2\baij-\bbi\bbj),\qquad\bsij=\f{1}{\bar b^2}(\bbi\bsj-\bbj\bsi).\label{specialoneformssqDb}
\end{eqnarray*}
In this case, $\bar b=b$ and $F$ is reexpressed as
\begin{eqnarray}\label{eimiignianieng}
F=\f{(\bar b\ba+\bb)^2}{\bar b^3\ba}.
\end{eqnarray}
\end{enumerate}
\end{theorem}

Theorem \ref{Douglasssq} is quite different from the corresponding result on regular square metrics, see Section \ref{iemieninigi} for reasons. The final classification result up to the classification for the $1$-forms satisfying (\ref{specialoneformssqDa}) and (\ref{specialoneformssqDb}). We are not sure that if it is possible to determine such $1$-forms completely for an abstract Riemannian metric. At this moment, we can only provide an  effective way to construct the required data. The key technique is a new kinds of metrical deformations called $\b$-deformations\cite{yct-dhfp}, which is also the main method for our whole problem.

\section{Preliminaries}
In local coordinates, the geodesics of a Finsler metric $F=F(x,y)$ are characterized by
\begin{eqnarray*}
\frac{d^{2}x^{i}}{dt^{2}}+2G^{i}\big(x,\frac{dx}{dt}\big)=0,
\end{eqnarray*}
where $G^{i}:=\frac{1}{4}g^{il}\left\{[F^{2}]_{x^{k}y^{l}}y^{k}-[F^{2}]_{x^{l}}\right\}$. The local function $G^{i}=G^{i}(x,y)$ define a global vector field $G:=y^{i}\frac{\partial}{\partial x^{i}}-2G^{i}\frac{\partial}{\partial y^{i}}$ on $TM\backslash \{0\}$, which is called the {\it spray} of $F$. For a Riemannian metric, the spray coefficients are determined by its Christoffel symbols as $G^{i}(x,y)=\frac{1}{2}\gamma^{i}_{jk}(x)y^{j}y^{k}$.

Douglas metrics can be characterized by
\begin{eqnarray}\label{Dou}
G^{i}=\frac{1}{2}\Gamma^{i}_{jk}(x)y^{j}y^{k}+P(x,y)y^{i},
\end{eqnarray}
where $\Gamma^{i}_{jk}(x)$ are local functions on $M$ and $P(x,y)$ is a local positively homogeneous function of degree one. It  is known that Douglas metrics can be also characterized by the following equations\cite{BM}
\begin{eqnarray}\label{31}
G^{i}y^{j}-G^{j}y^{i}=\frac{1}{2}(\Gamma^{i}_{kl}y^{j}-\Gamma^{j}_{kl}y^{i})y^{k}y^{l}.
\end{eqnarray}

By definition, a general $(\a,\b)$-metric is a Finsler metric expressed in the following form,
\begin{eqnarray*}
F=\a\phi(b^{2},s), ~~ s:=\frac{\b}{\a}, ~~ b^{2}:=\|\b\|^{2}_{\a},
\end{eqnarray*}
where $\phi(b^{2},s)$ is a smooth function defined on the domain $|s|\leq b<b_o$ for some positive number (maybe infinity) $b_o$.
$F$ is a regular Finsler metric for any Riemannian metric $\alpha$ and  any $1$-form $\beta$ if and only if $\phi(b^{2},s)$ satisfies
\begin{eqnarray}\label{ppp}
\p-s\pt>0,\quad\p-s\pt+(b^2-s^2)\ptt>0,
\end{eqnarray}
when $n\geq 3$ or
\begin{eqnarray}\label{ppp1}
\p-s\pt+(b^2-s^2)\ptt>0,
\end{eqnarray}
when $n=2$ \cite{yct-zhm-onan}.

Let $\alpha=\sqrt{a_{ij}(x)y^iy^j}$  and $\beta= b_i(x)y^i$.
Denote the coefficients of the covariant derivative of
$\b$ with respect to $\a$ by $b_{i|j}$, and let
\begin{eqnarray*}
&r_{ij}=\frac{1}{2}(b_{i|j}+b_{j|i}),~s_{ij}=\frac{1}{2}(b_{i|j}-b_{j|i}),
~r_{00}=r_{ij}y^iy^j,~s^i{}_0=a^{ij}s_{jk}y^k,&\\
&r_i=b^jr_{ji},~s_i=b^js_{ji},~r_0=r_iy^i,~s_0=s_iy^i,~r^i=a^{ij}r_j,~s^i=a^{ij}s_j,~r=b^ir_i,&
\end{eqnarray*}
where $(a^{ij}):=(a_{ij})^{-1}$ and $b^{i}:=a^{ij}b_{j}$. It is easy to see that $\b$ is closed if and only if $s_{ij}=0$.

According to \cite{yct-zhm-onan}, the spray coefficients $G^i$ of a general $(\alpha,\beta)$-metric $F=\a\phi(b^{2},s)$ are related to the spray coefficients ${}^\a G^i$ of
$\a$ and given by
\begin{eqnarray}\label{Gi}
G^i&=&{}^\a G^i+\a Q s^i{}_0+\left\{\Theta(-2\a Q s_0+r_{00}+2\a^2
R r)+\a\Omega(r_0+s_0)\right\}\frac{y^i}{\a}\nonumber\\
&&+\left\{\Psi(-2\a Q s_0+r_{00}+2\a^2 R
r)+\a\Pi(r_0+s_0)\right\}b^i -\a^2 R(r^i+s^i),
\end{eqnarray}
where
\begin{eqnarray*}
&\displaystyle Q=\frac{\pt}{\p-s\pt},\quad R=\frac{\phi_{1}}{\p-s\pt},\quad\Theta=\frac{(\p-s\pt)\pt-s\p\ptt}{2\p\big(\p-s\pt+(b^2-s^2)\ptt\big)},&\\
&\displaystyle\Psi=\frac{\ptt}{2\big(\p-s\pt+(b^2-s^2)\ptt\big)},\quad\Pi=\frac{(\p-s\pt)\pot-s\po\ptt}{(\p-s\pt)\big(\p-s\pt+(b^2-s^2)\ptt\big)},
\quad\Omega=\frac{2\phi_{1}}{\p}-\frac{s\p+(b^2-s^2)\pt}{\p}\Pi.&
\end{eqnarray*}

\section{Proof of Theorem \ref{characterizationforssqD}}
Firstly, we prove the necessity. By (\ref{31}) and (\ref{Gi}), we can prove
a general $\ab$-metric is a Douglas metric if and only if
\begin{eqnarray}\label{e41}
&&\a Q (s^i{}_0y^{j}-s^j{}_0y^{i})+\left\{\Psi(-2\a Q s_0+r_{00}+2\a^2 R
r)+\a\Pi(r_0+s_0)\right\}(b^iy^{j}-b^{j}y^{i})\nonumber\\
&&-\a^2 R[(r^i+s^i)y^{j}-(r^j+s^j)y^{i}]=\frac{1}{2}\big(T^{i}_{kl}y^{j}-T^{j}_{kl}y^{i}\big)y^{k}y^{l}
\end{eqnarray}
where $T^{i}_{kl}:=\Gamma^{i}_{kl}-\gamma^{i}_{kl}$ and $\gamma^{i}_{kl}:=\frac{\partial^{2}G^{i}_{\a}}{\partial y^{k}\partial y^{l}}$.
Especially, it follows from (\ref{e41}) that (\ref{ssq}) is a Douglas metric if and only if
\begin{eqnarray*}
&&12b\a^2(b^2\a^2-\b^2)(s^i{}_0y^{j}-s^j{}_0y^{i})+2\a^2\big\{b(b\a-\b)r_{00}-4b\a^2 s_{0}+2\a^3 r+\a(b\a-3\b)(r_0+s_0)\big\}(b^iy^{j}-b^{j}y^{i})\nonumber\\
&&-6\a^3(b^2\a^2-\b^2)\big\{(r^i+s^i)y^{j}-(r^j+s^j)y^{i}\big\}-3b(b\a-\b)(b^2\a^2-\b^2)\big(T^{i}_{kl}y^{j}-T^{j}_{kl}y^{i}\big)y^{k}y^{l}=0,
\end{eqnarray*}
which is equivalent to
\begin{eqnarray}\label{ratirrat}
b \Rat+\a\Irrat=0,
\end{eqnarray}
where $\Rat$ and $\Irrat$ given below are both polynomials of y
\begin{eqnarray*}
\Rat&:=&2\big\{6b^2(s^i{}_0y^{j}-s^j{}_0y^{i})+(r_{0}-3s_{0})(b^iy^{j}-b^{j}y^{i})\big\}\a^{4}-\b\big\{2r_{00}(b^iy^{j}-b^{j}y^{i})+12\b(s^i{}_0y^{j}
-s^j{}_0y^{i})\\
&&-3b^2\big(T^{i}_{kl}y^{j}-T^{j}_{kl}y^{i}\big)y^{k}y^{l}\big\}\a^2-3\b^3\big(T^{i}_{kl}y^{j}-T^{j}_{kl}y^{i}\big)y^{k}y^{l},\\
\Irrat&:=&2\big\{2r(b^iy^{j}-b^{j}y^{i})-3b^2[(r^i+s^i)y^{j}-(r^j+s^j)y^{i}]\big\}\a^{4}+\big\{2[b^{2} r_{00}-3\b(r_0+s_0)](b^iy^{j}-b^{j}y^{i})\\
&&+6\b^2[(r^i+s^i)y^{j}-(r^j+s^j)y^{i}]-3b^4\big(T^{i}_{kl}y^{j}-T^{j}_{kl}y^{i}\big)y^{k}y^{l}\big\}\a^2
+3b^2 \b^2\big(T^{i}_{kl}y^{j}-T^{j}_{kl}y^{i}\big)y^{k}y^{l}.
\end{eqnarray*}
Since $\a$ is irrational about $y$,  (\ref{ratirrat}) implies
\begin{eqnarray*}
\Rat=0, ~~ \Irrat=0.
\end{eqnarray*}
In particular, it follows from $b^2 \Rat+\b \Irrat=0$ that
\begin{eqnarray}\label{S2}
&&\a^2\big\{6b^4(s^i{}_0y^{j}-s^j{}_0y^{i})-3b^2\b[(r^i+s^i)y^{j}-(r^j+s^j)y^{i}]+[b^2 (r_{0}-3s_{0})+2\b r](b^iy^{j}-b^{j}y^{i})\big\}\nonumber\\
&&-3\b^2\big\{2b^2(s^i{}_0y^{j}-s^j{}_0y^{i})-\b[(r^i+s^i)y^{j}-(r^j+s^j)y^{i}]+(r_0+s_0)(b^iy^{j}-b^{j}y^{i})\big\}=0.
\end{eqnarray}
Contract (\ref{S2}) with $y_{i}$ yields
\begin{eqnarray}\label{S3}
&&2\b \big\{b^2 (r_{0}+3s_{0})-\b r \big\}y^{j}+\big\{b^2\a^2(r_{0}-3 s_{0})-3\b^2(r_0+s_0)+2\b \a^2 r\big\}b^{j}\nonumber\\
&&+3(b^2\a^2-\b^2)\big\{2b^2 s^j{}_0-\b (r^{j}+ s^{j})\big\}=0.
\end{eqnarray}
Furthermore, contract (\ref{S3}) with $b_{j}$ yields $(b^2\a^2-\b^2)(b^2 r_{0}+3b^2 s_{0}-\b r)=0$. Notice that only if the vector $y$ is parallel to $(b^{i})$, can $b^2\a^2-\b^2$ be equal to zero. Hence, when $y\neq \lambda (b^{i})$, $b^2 r_{0}+3b^2 s_{0}-\b r=0$. It follows from continuity that $b^2 r_{0}+3b^2 s_{0}-\b r\equiv 0$ for any $y\in T_xM$, so
\begin{eqnarray}\label{S5}
s_{0}=\frac{1}{3b^2}(\b r-b^2 r_{0})
\end{eqnarray}
and hence $s^{i}=\frac{1}{3b^2}( r b^{i}-b^2 r^{i})$. Plugging them into (\ref{S3}) yields $s^j{}_0=\frac{1}{3b^2}(\b r^{j}-r_{0}b^{j})$, i.e.,
\begin{eqnarray}\label{S10}
\sij=-\frac{1}{3b^2}(\bi\rj-\bj\ri).
\end{eqnarray}
As a result, $\Irrat$ becomes
\begin{eqnarray}\label{S11}
 2\a^2(b^2 r_{00}-2\b r_{0}+\a^2 r)(b^{i}y^{j}-b^{j}y^{i})=(b^2\a^2-\b^2)\big\{4\a^2(r^{i}y^{j}-r^{j}y^{i})+3b^2\big(T^{i}_{kl}y^{j}-T^{j}_{kl}y^{i}\big)y^{k}y^{l}\big\}.
\end{eqnarray}
Notice that the rank of the matrix $\{b^2\aij-\bi\bj\}$ is equal to $n-1$, so $b^2\a^2-\b^2$ is an irreducible quadratic polynomial when $n\geq3$. As a result, it can't be divided by $\a^2$ and the linear function $b^{i}y^{j}-b^{j}y^{i}$. Hence, there must be a scalar function $\sigma(x)$ such that $b^2 r_{00}-2\b r_{0}+\a^2 r=\sigma(b^2\a^2-\b^2)$,
or equivalently,
\begin{eqnarray}\label{S13}
 b^4 r_{00}-2b^2\b r_{0}+\b^2 r=(\sigma b^2-r)(b^2\a^2-\b^2).
\end{eqnarray}
Be attend that the argument above depends on the condition $n\geq3$. When $n=2$,  the linear function $b^{i}y^{j}-b^{j}y^{i}$ can possible become the factor of $b^2\a^2-\b^2$. For instance, take $\a=\sqrt{u^2+v^2}$, $\b=vx-uy$, then $b^1y^2-b^2y^1=ux+vy$ is the factor of $b^2\a^2-\b^2=(ux+uy)^2$. Hence, we are not sure that if the key condition (\ref{S13}) still holds for the $2$-dimensional case.

By (\ref{S13}) one can see that $r_{00}$ is given in the form
\begin{eqnarray}\label{S14}
r_{00}=c(x)\a^2+d(x)\b^2+2\b\theta
\end{eqnarray}
for some scalar function $c(x)$, $d(x)$ and some $1$-form $\theta$. Conversely, it is easy to verify that any $1$-form satisfying (\ref{S14}) must satisfies (\ref{S13}). Moreover, we can always assume that $\theta$ is perpendicular to $\b$, i.e., $\theta_ib^i=0$. That is because if $\theta$ is not perpendicular to $\b$, we can replay $\theta$ as $\theta'+\frac{\theta^lb_l}{b^2}\b$, then the new $1$-form $\theta'$ will be perpendicular to $\b$.

From now on, we will assume $\roo$ is given by (\ref{S14}) in which $\theta$ is perpendicular to $\b$. In this case, we have
\begin{eqnarray*}
\ro=(c+b^2 d)\b+b^2\theta,\qquad r=(c+b^2 d)b^2.
\end{eqnarray*}
Plugging them into (\ref{S5}) yields $\theta_{i}=-\frac{3}{b^2}s_{j}$. Hence, $\rij$ is given by (\ref{ssqrij}) and by (\ref{S10}) we have (\ref{ssqsij}).

Conversely, if (\ref{ssqrij}) and (\ref{ssqsij}) hold, by (\ref{Gi}) we have
\begin{eqnarray*}
G^{i}= {}^\a G^i+Py^{i}-\left(\frac{1}{3} d b^{i}-\frac{2}{b^2}s^{i}\right)\a^{2},
\end{eqnarray*}
where  $P:=\frac{\a}{3b(b^2-s^2)}\left\{b(b-2s)(c+ds^2)+(2b^2+2bs-3s^2)(c+b^2d)\right\}-\frac{4}{b^2}s_{0}$.
Hence, $F$ is a Douglas metric due to (\ref{Dou}).

\section{Deformations}\label{iemieninigi}
In the rest of this paper, we would like to construct as many non-trivial Douglas singular square metrics as possible. The key technique is a special kind of metrical deformations called $\b$-deformations.

$\b$-{\emph deformations}, first introduced by the first author in \cite{yct-dhfp}, are a triple kind of deformations in terms of a given Riemannian metric $\a$ and a $1$-form $\b$ as follows,
\begin{eqnarray*}
&\ta=\sqrt{\a^2-\kappa(b^2)\b^2},\qquad\tb=\b;\label{b1}\\
&\ha=e^{\rho(b^2)}\ta,\qquad\hb=\tb;\label{b2}\\
&\ba=\ha,\qquad\bb=\nu(b^2)\hb.\label{b3}
\end{eqnarray*}
Be attention that the factor $\kappa$ must satisfy an additional condition $1-\kappa b^2>0$ to keep the resulting metric positive definite.

According to Lemma 2, Lemma 3 and Lemma 4 in \cite{yct-dhfp}, we have a basic formula for $\b$-deformations immediately.
\begin{proposition}\label{relationunderdeformationssg}
After $\b$-deformations,
\begin{eqnarray}
\brij&=&\f{\nu}{1-\kappa b^2}\rij+\f{\kappa\nu}{1-\kappa b^2}(\bi\sj+\bj\si)-\f{\kappa'\nu}{1-\kappa b^2}r\bi\bj+\f{2\rho'\nu}{1-\kappa b^2}r(\aij-\kappa\bi\bj)\nonumber\\
&&+\left(\f{\kappa'\nu b^2}{1-\kappa b^2}-2\rho'\nu+\nu'\right)\left\{\bi(\rj+\sj)+\bj(\ri+\si)\right\},\label{iwneingien}\\
\bsij&=&\nu\sij+\nu'\left\{\bi(\rj+\sj)-\bj(\ri+\si)\right\}.\label{iwneingienn}
\end{eqnarray}
\end{proposition}
Be attention that $\brij:=\frac{1}{2}(\bbij+\bar b_{j|i})$ and $\bsij:=\frac{1}{2}(\bbij-\bar b_{j|i})$ in which $\bbij$ means the covariant derivative of $\bb$ with respect to the corresponding metric $\ba$.

As a application of Proposition \ref{relationunderdeformationssg}, the following result shows the relationship between singular square metrics and the metrics $F=\sq$, although it is useless in the next discussions.
\begin{proposition}\label{imineiigadiing}
Suppose $\a$ and $\b$ satisfy (\ref{ssqrij}) and (\ref{ssqsij}). Take $\kappa=0$, $\rho=\ln(b^2)$ and $\nu=b$ as the deformations factors, then the resulting data $\ba$ and $\bb$ satisfy
\begin{eqnarray}\label{mieningienigndl}
\brij=\bar\tau(\baij-\bbi\bbj),\qquad\bsij=0,
\end{eqnarray}
where $\bar\tau=\frac{3c+2b^2d}{b^3}$. In this case, $\bar b=1$ and the metric $F=\ssq$ is given by $F=\frac{(\ba+\bb)^2}{\ba}$.
\end{proposition}

According to Theorem 1.1 in \cite{Li}, the necessary and sufficient condition for a square metric $F=\sq$ to be of vanishing Douglas curvature is that $\a$ and $\b$ satisfy
\begin{eqnarray*}
\bij=\tau\left\{(1+2b^2)a_{ij}-3b_ib_j\right\}.
\end{eqnarray*}
One can see that (\ref{mieningienigndl}) is coincide with the above result by taking $b=1$, in spite of Theorem 1.1 in \cite{Li} was obtained for the regular case~(namely $b<1$).

Proposition \ref{imineiigadiing} indicates that all the singular square metrics $F=\ssq$ belong to the metrical category $F=\frac{(\ba+\bb)^2}{\ba}$ with $\bar b=1$. Conversely, any Finsler metric $F=\frac{(\ba+\bb)^2}{\ba}$ with $\bar b=1$ can always be reexpressed as the form (\ref{ssq}) such that $b$ is not a constant. Let
\begin{eqnarray*}
\a=\f{\ba}{b(x)^2},\qquad\b=\f{\bb}{b(x)},
\end{eqnarray*}
where $b(x)$ is an arbitrarily positive scalar function, then it is easy to verify that $b(x)$ is  the length of the new $1$-form $\b$ with respect to the new metric $\a$. In this case, $F=\ssq$. Anyway, singular square metrics are just the Finsler metrics $F=\frac{(\ba+\bb)^2}{\ba}$ with $\bar b=1$.

The condition (\ref{mieningienigndl}) is simpler than (\ref{ssqrij}) and (\ref{ssqsij}) in some sense. However, it seems that it is impossible to discuss singular square metrics in the form $F=\sq$ using $\b$-deformations. The reason is $b=1$, which leads to the irreversibility of the $\b$-deformations using in Proposition \ref{imineiigadiing}.

Moreover, the observation above indicate a fact that any given singular square metric $F=\frac{(\ba+\bb)^2}{\ba}$ with $\bar b=1$ has infinity many ways to be expressed as the form (\ref{ssq}):~a different $b(x)$ means a different way. Actually, such phenomenon can be described in another way.
\begin{lemma}\label{Le1}
After $\b$-deformations, $F=\ssq$ is still of the form $F=\frac{(\bar b\ba+\bb)^2}{\ba}$  if and only if
\begin{eqnarray}\label{specialdeformforssq}
\kappa=0,\qquad\nu=e^{\frac{1}{2}\rho}.
\end{eqnarray}
\end{lemma}
\begin{proof}
By (\ref{s4}) we have
\begin{eqnarray}\label{s8}
F=\f{\left(\frac{e^{-\rho}|\nu|}{\sqrt{1-\kappa b^2}}b\cdot e^{\rho}\a+\nu\b\right)^2}{e^\rho\a}=\f{\nu^2\left(\frac{1}{\sqrt{1-\kappa b^2}}b\a+\mathrm{sgn}\,(\nu)\b\right)^2}{e^\rho\a}.
\end{eqnarray}
Comparing (\ref{s8}) with  $F=\ssq$ we obtain the conclusion.
\end{proof}

Although the form (\ref{ssq}) for singular square metrics has some shortcoming, we have to use it in the whole discussions. The key point is that the method of $\b$-deformations works because $b$ is not necessary a constant.

The following result shows the effect of the $\b$-deformations satisfying (\ref{specialdeformforssq}).
\begin{lemma}\label{Le}
If $\a$ and $\b$ satisfy (\ref{ssqrij}), then after the $\b$-deformations satisfying (\ref{specialdeformforssq}),
\begin{eqnarray}\label{owmeingeninnj}
\brij=\bar c(x)\baij+\bar d(x)\bbi\bbj-\f{3}{\bar b^2}(\bbi\bsj+\bbj\bsi),
\end{eqnarray}
where
\begin{eqnarray*}
\bar c=e^{-\frac{3}{2}\rho}\left\{(1+2b^2\rho')c+2b^4\rho'd\right\},\qquad
\bar d=-e^{-\frac{1}{2}\rho}\left\{3\rho'c-(1-3b^2\rho')d\right\}.
\end{eqnarray*}
\end{lemma}
\begin{proof}
Now that $F$ is still expressed as the form $F=\frac{(\bar b\ba+\bb)^2}{\ba}$ after the $\b$-deformations satisfying (\ref{specialdeformforssq}), then Theorem \ref{characterizationforssqD} indicate that $\brij$ must satisfy (\ref{owmeingeninnj}) for some function $\bar c(x)$ and $\bar d(x)$. Finally, $\bar c$ and $\bar d$ can be determined by (\ref{ssqrij}) and (\ref{iwneingien}) directly.
\end{proof}

The property (\ref{ssqsij}) is interesting. It is invariable under all $\b$-deformations.

\begin{lemma}\label{inv}
If $\a$ and $\b$ satisfies (\ref{ssqsij}), then after any $\b$-deformation,
\begin{eqnarray*}
\bsij=\f{1}{\bar b^2}(\bbi\bsj-\bbj\bsi).
\end{eqnarray*}
\end{lemma}
\begin{proof}
It is easy to verify that $\bar{b}^{i}=\frac{\nu e^{-2\rho}}{1-\kappa b^2}b^{i}$, so
\begin{eqnarray}\label{s4}
\bar{b}^{2}=\frac{e^{-2\rho}\nu^{2}}{1-\kappa b^2}b^{2}
\end{eqnarray}
and by (\ref{iwneingienn}) we have
\begin{eqnarray*}
\bar{s}_{i}=\frac{\nu e^{-2\rho}}{1-\kappa b^2}\big\{(\nu+b^2\nu')s_{i}+\nu'(b^2r_{i}-rb_{i})\big\}.
\end{eqnarray*}
Hence, by (\ref{iwneingienn}) again
\begin{eqnarray}\label{s5}
\bsij-\f{1}{\bar b^2}(\bbi\bsj-\bbj\bsi)=\nu\left\{\sij-\f{1}{b^2}(\bi\sj-\bj\si)\right\}.
\end{eqnarray}
By (\ref{ssqsij}) and (\ref{s5}), we complete the proof.
\end{proof}

According to Lemma \ref{Le1}, one can see that for a specific singular square metrics, there are infinity many way to express it to be the form $F=\ssq$. Such phenomenon does not happen for regular square metrics:~for a specific regular square metric, there is only one way to express it to be the form $F=\sq$. Hence, we need to differentiate the repetitive metrics due to Lemma \ref{Le}.

In the next two sections, we will discuss this problem under the assumption that the ratio $c:d$ depends only on $b$. We will show that, the case $c+b^2d\neq0$ can always be reduced to the case $d=0$, and the case $c+b^2d=0$ is self-consistent under the allowable deformations.

\section{Case 1:~$c+b^2d\neq0$}
\begin{lemma}\label{ieminignienigss}
Suppose that $\a$ and $\b$ satisfy (\ref{ssqrij}) under the assumption that $c:d$ depends only on $b$. If $c+b^2d\neq0$, then $\ba$ and $\bb$ satisfy
\begin{eqnarray}\label{s9}
\brij=\bar c(x)\baij-\f{3}{\bar b^2}(\bbi\bsj+\bbj\bsi)
\end{eqnarray}
when $\kappa=0$, $\rho=\frac{1}{3}\int\f{d}{c+b^2d}\ud b^2$ and $\nu=e^{\frac{1}{2}\rho}$.
\end{lemma}
\begin{proof}
According to Lemma \ref{Le}, (\ref{s9}) holds if
\begin{eqnarray*}
3\rho'c-(1-3b^2\rho')d=0,
\end{eqnarray*}
Since  $c+b^2d\neq0$, solving the above equation yields $\rho=\frac{1}{3}\int\f{d}{c+b^2d}\ud b^2$.
\end{proof}

\begin{lemma}\label{equationXY}
If $\a$ and $\b$ satisfy (\ref{ssqrij}), then after $\b$-deformations,
\begin{eqnarray}\label{s10}
\brij=\f{e^{-2\rho}\nu}{1-b^2\kappa}\left\{(1+2b^2\rho')c+2b^4\rho'd\right\}\baij+X\bi\bj+Y(\bi\sj+\bj\si),
\end{eqnarray}
where
\begin{eqnarray*}
X&=&\f{\nu}{1-b^2\kappa}\left\{(c\kappa+d)+b^2(c+b^2d)\kappa'\right\}-2(c+b^2d)(2\rho'\nu-\nu'),\\
Y&=&-\f{\nu}{b^2(1-b^2\kappa)}(3-b^2\kappa+2b^4\kappa')+4\rho'\nu-2\nu'.
\end{eqnarray*}
In particular, if the deformation factors satisfy
\begin{eqnarray}\label{s14}
\kappa=\f{1}{b^2}-Cb^2e^{\int\f{c}{b^2(c+b^2d)}\,\ud b^2},\qquad\nu=\f{D(1-b^2\kappa)e^{2\rho}}{b^3},
\end{eqnarray}
where $C>0$ and $D\neq0$ are constants, then $\bb$ is conformal with respect to $\ba$.
\end{lemma}
\begin{proof}
If $\b$ satisfies (\ref{ssqrij}), then
\begin{eqnarray}\label{s13}
\ri=(c+b^2 d)\bi-3\si,\qquad r=(c+b^2 d)b^{2}.
\end{eqnarray}
Plugging (\ref{ssqrij}) and (\ref{s13}) into (\ref{iwneingien}) yields (\ref{s10}).

By solving the equations $X=0$ and $Y=0$ we obtain (\ref{s14}), and the role of deformation factors asks $C>0$ and $D\neq0$.
\end{proof}

\begin{proposition}\label{meimeinignni}
$\a$ and $\b$ satisfy (\ref{ssqrij}) and (\ref{ssqsij}) with $d=0$ if and only if
\begin{eqnarray*}
\ba:=\sqrt{\a^2-(b^{-2}-b^4)\b^2},\qquad\bb:=b^3\b
\end{eqnarray*}
satisfy
\begin{eqnarray*}
\brij=\bar\tau(x)\baij,\qquad\bsij=\f{1}{\bar b^2}(\bbi\bsj-\bbj\bsi).
\end{eqnarray*}
\end{proposition}
\begin{proof}
Taking $C=1$, $D=1$ and $\rho=0$ in (\ref{s14}) yields $\kappa=b^{-2}-b^4$ and $\nu=b^3$. So $\ba$ and $\bb$ satisfy $\brij=\bar\tau\baij$ due to Lemma \ref{equationXY}. And $\bsij=\f{1}{\bar b^2}(\bbi\bsj-\bbj\bsi)$ due to Lemma \ref{inv}. The sufficiency holds since the deformations used here is reversible.
\end{proof}

\section{Case 2:~$c+b^2d=0$}
The case $c+b^2d=0$ is special, because it has some invariance under $\b$-deformations. As a deduction of Lemma \ref{Le}, we have the below fact.
\begin{lemma}\label{iemimaubgeg}
If $\a$ and $\b$ satisfy (\ref{ssqrij}) with $c+b^2d=0$, which means
\begin{eqnarray}\label{s17}
\rij=-d(x)(b^2\aij-\bi\bj)-\f{3}{b^2}(\bi\sj+\bj\si),
\end{eqnarray}
then $\ba$ and $\bb$ satisfies
\begin{eqnarray*}
\brij=-d(x)e^{-\f{1}{2}\rho}(\bar b^2\baij-\bbi\bbj)-\f{3}{\bar b^2}(\bbi\bsj+\bbj\bsi)
\end{eqnarray*}
after $\b$-deformations satisfying (\ref{specialdeformforssq}).
\end{lemma}

According to Lemma \ref{equationXY}, we can see that
\begin{eqnarray*}
\brij=-d\nu^{-1}(\bar b^2\baij-\bbi\bbj)+Y(\bi\sj+\bj\si)
\end{eqnarray*}
when $c+b^2d=0$. Hence, it is impossible to make $\bb$ to be conformal with respect to $\ba$. The only thing we can do is to dispel the item $\bi\sj+\bj\si$ by taking
\begin{eqnarray}\label{iemigneninl}
\nu=Db^{-3}(1-b^2\kappa)e^{2\rho}.
\end{eqnarray}

\begin{proposition}\label{iwienignubhd}
$\a$ and $\b$ satisfy (\ref{ssqrij}) and (\ref{ssqsij}) with $c+b^2d=0$ if and only if
\begin{eqnarray*}
\ba:=b^3\a,\qquad\bb:=b^3\b
\end{eqnarray*}
satisfy
\begin{eqnarray*}
\brij=\bar\tau(x)(\bar b^2\baij-\bbi\bbj),\qquad\bsij=\f{1}{\bar b^2}(\bbi\bsj-\bbj\bsi).
\end{eqnarray*}
\end{proposition}
\begin{proof}
It is obvious that $\kappa=0$, $\rho=\ln b^3$ and $\nu=b^3$ satisfy (\ref{iemigneninl}). So $\ba$ and $\bb$ satisfy $\brij=\bar\tau(\bar b^2\baij-\bbi\bbj)$, and $\bsij=\f{1}{\bar b^2}(\bbi\bsj-\bbj\bsi)$ due to Lemma \ref{inv}. The sufficiency holds since the deformation used here is reversible.
\end{proof}

\section{Required data and examples}
\begin{theorem}\label{eipomwming}
If $\a$ and $\b$ satisfy
\begin{eqnarray}\label{iiiiiweing}
\rij=\tau(x)\aij,\qquad\sij=\f{1}{b^2}(\bi\sj-\bj\si),
\end{eqnarray}
then $\ba$ and $\bb$ satisfy the same property after $\b$-deformations, i.e.,
\begin{eqnarray}\label{imeiigninissaas}
\brij=\bar\tau(x)\baij,\qquad\bsij=\f{1}{\bar b^2}(\bbi\bsj-\bbj\bsi),
\end{eqnarray}
if and only if
\begin{enumerate}[(a)]
\item when $\b$ is Killing but not closed,
\begin{eqnarray}\label{omiiiengbiba}
\nu=C(1-b^2\kappa)e^{2\rho},
\end{eqnarray}
in which $C\neq0$ is a constant. In this case, $\bb$ must be Killing, and it is not closed unless $(1-b^2\kappa)e^{2\rho}=\f{D}{b^2}$ for some positive constant $D$;
\item when $\b$ is not Killing but closed,
\begin{eqnarray*}
\nu=C\sqrt{1-b^2\kappa}e^{2\rho},
\end{eqnarray*}
in which $C\neq0$ is a constant. In this case, $\bb$ must be closed, and it is not Killing unless $e^{2\rho}=\f{D}{b^2}$ for some positive constant $D$;
\item when $\b$ is neither Killing nor closed,
\begin{eqnarray*}
\kappa=\f{C}{b^2},\qquad\nu=De^{2\rho},
\end{eqnarray*}
in which $C<1$ and $D\neq0$ are constants. In this case, $\bb$ is neither Killing nor closed unless $e^{2\rho}=\f{E}{b^2}$ for some positive constant $E$.
\end{enumerate}
\end{theorem}
\begin{proof}
It is easy to see that Case (a) and Case (b) imply Case (c). On the other hand, Case (b) had been prove in \cite{yct-dmab}. So we just need to prove Case (a) here.

Assume (\ref{iiiiiweing}) holds. By Proposition \ref{relationunderdeformationssg} we have
\begin{eqnarray*}
\brij=\f{\tau e^{-2\rho}(1+2b^2\rho')\nu}{1-b^2\kappa}\baij+\tau A\bi\bj+B(\bi\sj+\bj\si),
\end{eqnarray*}
where
\begin{eqnarray*}
A=\f{(\kappa+b^2\kappa')\nu}{1-b^2\kappa}-4\rho'\nu+2\nu',\qquad B=\f{(\kappa+b^2\kappa')\nu}{1-b^2\kappa}-2\rho'\nu+\nu'.
\end{eqnarray*}

If $\b$ is Killing but not closed, which means that $\tau=0$ and $\so\neq0$, then by solving the equation $B=0$ we obtain (\ref{omiiiengbiba}). By (\ref{iwneingienn}) we can see that
\begin{eqnarray*}
\bsij=\f{C}{b^2}e^{2\rho}\left\{1-2b^2\kappa-b^4\kappa'-2b^2(1-b^2\kappa)\rho'\right\}(\bi\sj-\bj\si)
\end{eqnarray*}
in this case. Hence, $\bb$ is not closed unless $1-2b^2\kappa-b^4\kappa'-2b^2(1-b^2\kappa)\rho'=0$, which implies $(1-b^2\kappa)e^{2\rho}=\f{D}{b^2}$.
\end{proof}
It should be make attention that, when  (\ref{omiiiengbiba}) holds, then the condition $(1-b^2\kappa)e^{2\rho}=\f{D}{b^2}$ indicates $\bb$ has constant length, and vice versa. In another word, $\bb$ is Killing but not closed as long as $\bar b$ is not a constant. Case (b) and (c) are similar.

\begin{theorem}\label{eiiweggga}
If $\a$ and $\b$ satisfy (\ref{iiiiiweing}), then $\ba$ and $\bb$ satisfy
\begin{eqnarray}\label{fwemaimieng}
\brij=\bar\tau(x)(\bar b^2\baij-\bbi\bbj),\qquad\bsij=\f{1}{\bar b^2}(\bbi\bsj-\bbj\bsi),
\end{eqnarray}
after $\b$-deformations if and only if
\begin{enumerate}[(a)]
\item when $\b$ is Killing but not closed,
\begin{eqnarray*}
\nu=C(1-b^2\kappa)e^{2\rho},
\end{eqnarray*}
in which $C\neq0$ is a constant. In this case, $\bb$ must be Killing, and it is not closed unless $(1-b^2\kappa)e^{2\rho}=\f{D}{b^2}$ for some positive constant $D$;
\item when $\b$ is not Killing but closed,
\begin{eqnarray}\label{omiiiengbibb}
\nu=\f{C\sqrt{1-b^2\kappa}e^{\rho}}{b},
\end{eqnarray}
in which $C\neq0$ is a constant. In this case, $\bb$ must be closed and of constant length, and it is not parallel unless $e^{2\rho}=\f{D}{b^2}$ for some positive constant $D$;
\item when $\b$ is neither Killing nor closed,
\begin{eqnarray*}
\rho=-\f{1}{2}\ln(1-b^2\kappa)-\f{1}{2}\ln b^2+C,\qquad\nu=\f{D}{b^2},
\end{eqnarray*}
in which $C$ and $D\neq0$ are constants. In this case, $\bb$ is must be closed and of constant length, and it is not parallel unless $\kappa=\f{E}{b^2}$ for some constant $E<1$.
\end{enumerate}
\end{theorem}
\begin{proof}
It is easy to see that Case (a) and Case (b) imply Case (c). On the other hand, Case (a) is just the Case (a) in Theorem \ref{eipomwming}. So we just need to proof Case (b) here.

Assume (\ref{iiiiiweing}) holds. By Proposition \ref{relationunderdeformationssg} we have
\begin{eqnarray*}
\brij=\f{\tau (1+2b^2\rho')}{b^2\nu}(\bar b^2\baij-\bbi\bbj)+\tau A\bi\bj+B(\bi\sj+\bj\si),
\end{eqnarray*}
where
\begin{eqnarray*}
A=\f{(1+b^4\kappa')\nu}{b^2(1-b^2\kappa)}-2\rho'\nu+2\nu',\qquad B=\f{(\kappa+b^2\kappa')\nu}{1-b^2\kappa}-2\rho'\nu+\nu'.
\end{eqnarray*}

If $\b$ is not Killing but closed, which means that $\tau\neq0$ and $\so=0$, then by solving the equation $A=0$ we obtain (\ref{omiiiengbibb}). In this case,
\begin{eqnarray*}
\bar b^2=\f{\nu^2e^{-2\rho}}{1-b^2\kappa}b^2=C^2,
\end{eqnarray*}
so $\bar b$ must be a constant.

Finally, it is obviously that $\brij\neq0$ unless $1+2b^2\rho'=0$, which implies $e^{2\rho}=\f{D}{b^2}$.
\end{proof}

\begin{example}\label{oubbfhe2}
$\a=|y|$ and $\b=\xy$ satisfy (\ref{iiiiiweing}) with $\sij=0$. By Theorem \ref{eipomwming} (b) we know that
\begin{eqnarray}\label{pwieimign}
\ba=e^\rho\sqrt{|y|^2-\kappa\xy^2},\qquad\bb=C\sqrt{1-\kappa\xx}e^{2\rho}\xy
\end{eqnarray}
satisfy (\ref{imeiigninissaas}) with $\bsij=0$, where $\kappa=\kappa(\xx)$ and $\rho=\rho(\xx)$ are two arbitrary functions with an additional condition $1-\kappa\xx>0$. Moreover, $\bar b^2=C^2e^{2\rho}\xx$. As a result, one can obtain many Douglas metrics by (\ref{iemiingiiiwa}). However, since $\bb$ is closed, there is a simper way to obtain them. According to Theorem \ref{characterizationforssqD}, when $\b$ is closed and conformal, then $F=\ssq$ is a Douglas metric. That is to say, any excess deformations is not necessary in this case.  Hence, $F=\frac{(\bar b\ba+\bb)^2}{\ba}$ is a Douglas metric when $\ba$ and $\bb$ are given by (\ref{pwieimign}). In particular, take $C=1$ and $\kappa=\frac{\mu}{1+\mu\xx}$, then the following singular square metrics
\begin{eqnarray*}
F=\frac{\left(|x|\sqrt{(1+\mu\xx)\yy-\mu\xy^2}+\xy\right)^2}{\sqrt{1+\mu\xx}\cdot\sqrt{(1+\mu\xx)\yy-\mu\xy^2}}
\end{eqnarray*}
are Douglas metrics.
\end{example}

\begin{example}
$\a=|y|$ and $\b=\xy$ satisfy (\ref{iiiiiweing}) with $\sij=0$. By Theorem \ref{eiiweggga} (b) we know that
\begin{eqnarray}\label{pwieimignbb}
\ba=e^\rho\sqrt{|y|^2-\kappa\xy^2},\qquad\bb=C\sqrt{1-\kappa\xx}e^{\rho}\f{\xy}{|x|}
\end{eqnarray}
satisfy (\ref{fwemaimieng}) with $\bsij=0$. Moreover, $\bar b^2=C^2$. By  (\ref{eimiignianieng}) we can obtain many Douglas metrics. In particular, take $C=1$, $\kappa=\frac{\mu}{1+\mu\xx}$ and $\rho=0$, then the following singular square metrics
\begin{eqnarray*}
F=\frac{\left(|x|\sqrt{(1+\mu\xx)\yy-\mu\xy^2}+\xy\right)^2}{\xx\sqrt{1+\mu\xx}\cdot\sqrt{(1+\mu\xx)\yy-\mu\xy^2}}
\end{eqnarray*}
are Douglas metrics.
\end{example}

\begin{example}\label{owmienina}
$\a=|y|$ and $\b=x^2y^1-x^1y^2$ satisfy (\ref{iiiiiweing}) with $\rij=0$. By Theorem \ref{eiiweggga} (a) we know that
\begin{eqnarray*}
\ba=e^\rho\sqrt{|y|^2-\kappa\cdot(x^2y^1-x^1y^2)^2},\qquad\bb=C\{1-\kappa\cdot[(x^1)^2+(x^2)^2]\}e^{2\rho}(x^2y^1-x^1y^2)
\end{eqnarray*}
satisfy (\ref{fwemaimieng}) with $\bar\tau=0$, where $\kappa=\kappa(t)$ and $\rho=\rho(t)$ are two arbitrary functions with an additional condition $1-\kappa t>0$ in which $t:=(x^1)^2+(x^2)^2$. Moreover, $\bar b^2=C^2(1-\kappa t)e^{2\rho}t$. By (\ref{eimiignianieng}) we can obtain many Douglas metrics. In particular, Take $\kappa=0$, $\rho=0$ and $\nu=1$, then the following singular square metric
\begin{eqnarray*}
F=\f{\left(\sqrt{(x^1)^2+(x^2)^2}|y|+x^2y^1-x^1y^2\right)^2}{\sqrt{[(x^1)^2+(x^2)^2]^3}|y|}.
\end{eqnarray*}
is a Douglas metric.
\end{example}

(\ref{iemiingiiiwa}) does not been used to construct any example above. As we have mentioned in Example \ref{oubbfhe2} and Example \ref{owmienina}, if $\b$ is closed and conformal, we can use the original expression $F=\ssq$ directly, and if $\b$ is Killing, then we can use the expression (\ref{eimiignianieng}). Both of them are simpler than (\ref{iemiingiiiwa}). However, (\ref{iemiingiiiwa}) is useful when $\b$ satisfies (\ref{specialoneformssqDa}) but neither closed nor Killing, although we are not sure that whether such $1$-forms exist or not. Finally, we would like to point out a fact without argument to end this paper that there is not any $1$-form satisfy (\ref{specialoneformssqDa}) which is neither closed nor Killing when $\a$ is flat.

\noindent Changtao Yu\\
School of Mathematical Sciences, South China Normal
University, Guangzhou, 510631, P.R. China\\
aizhenli@gmail.com
\newline
\newline
\newline
\noindent Hongmei Zhu\\
College of Mathematics and Information Science, Henan Normal University, Xinxiang, 453007, P.R. China\\
zhm403@163.com
\end{document}

%% file: symbol.tex
\newcommand{\ud}{\mathrm{d}}

\newcommand{\f}{\frac}

\newcommand{\xx}{|x|^2}
\newcommand{\yy}{|y|^2}
\newcommand{\xy}{\langle x,y\rangle}

\newcommand{\pppp}[4]%
  {\frac{\partial^3{#1}}{\partial{#2}\partial{#3}\partial{#4}}}

\newcommand{\p}{\phi}

\newcommand{\pt}{\phi_2}
\newcommand{\ptt}{\phi_{22}}
\newcommand{\pot}{\phi_{12}}

\newcommand{\sq}{\frac{(\alpha+\beta)^2}{\alpha}}

\renewcommand{\a}{\alpha}
\renewcommand{\b}{\beta}
\newcommand{\ab}{(\alpha,\beta)}
\newcommand{\ta}{\tilde\alpha}
\newcommand{\tb}{\tilde\beta}
\newcommand{\ha}{\hat\alpha}
\newcommand{\hb}{\hat\beta}
\newcommand{\ba}{\bar\alpha}
\newcommand{\bb}{\bar\beta}

\newcommand{\aij}{a_{ij}}

\newcommand{\bij}{b_{i|j}}

\newcommand{\baij}{\bar a_{ij}}

\newcommand{\bbij}{\bar b_{i|j}}
\newcommand{\bi}{b_i}
\newcommand{\bj}{b_j}

\newcommand{\bbi}{\bar b_i}
\newcommand{\bbj}{\bar b_j}

\newcommand{\rij}{r_{ij}}

\newcommand{\brij}{\bar r_{ij}}

\newcommand{\roo}{r_{00}}

\newcommand{\ri}{r_i}
\newcommand{\rj}{r_j}

\newcommand{\ro}{r_0}

\newcommand{\sij}{s_{ij}}

\newcommand{\bsij}{\bar s_{ij}}
\newcommand{\si}{s_i}
\newcommand{\sj}{s_j}

\newcommand{\bsi}{\bar s_i}

\newcommand{\bsj}{\bar s_j}
\newcommand{\so}{s_0}

\newcommand{\po}{p_0}

\newcommand{\Rat}{\mathfrak{Rat}}
\newcommand{\Irrat}{\mathfrak{Irrat}}

\newcommand{\ssq}{\frac{(b\alpha+\beta)^2}{\alpha}}